\documentclass[runningheads]{llncs}
\usepackage{graphicx}
\usepackage{amsmath}
\usepackage{amssymb}
\usepackage{amsfonts}
\usepackage{graphicx}
\usepackage{hyperref}
\hypersetup{
    colorlinks=true,
    linkcolor=blue,
    filecolor=magenta,      
    urlcolor=cyan,
}

\usepackage{algorithmic}
\usepackage{algorithm}
\usepackage{tikz-cd}
\usepackage{chngcntr}

\usepackage{pgf, tikz}
\usetikzlibrary{arrows, automata}
\makeatletter
\def\Ddots{\mathinner{\mkern1mu\raise\p@
\vbox{\kern7\p@\hbox{.}}\mkern2mu
\raise4\p@\hbox{.}\mkern2mu\raise7\p@\hbox{.}\mkern1mu}}
\makeatother

\newcommand{\agcd}[2]{\ensuremath{\mathrm{agcd}_{\rho}^{\sigma}(#1,#2)}}

\usepackage{xcolor}

\def\mp{MAPLE }
\newcommand*{\QEDB}{\hfill\ensuremath{\square}}

\usepackage[T1]{fontenc}
\usepackage[utf8]{inputenc}
\usepackage{lmodern}

\begin{document}
\title{Approximate \texttt{GCD} in a Bernstein basis}
%
%
\author{Robert M. Corless \\
Leili Rafiee Sevyeri}
\authorrunning{Robert M. Corless, Leili Rafiee Sevyeri}
%
\institute{Ontario Research Centre for Computer Algebra, School of Mathematical and Statistical Sciences, Western University.\\
\and
\email{rcorless@uwo.ca}\\
\and 
\institute{Ontario Research Centre for Computer Algebra, School of Mathematical and Statistical Sciences, Western University. \\
 \and
\email{rcorless@uwo.ca\\ lrafiees@uwo.ca}
}}
\maketitle              
\begin{abstract}
We adapt Victor Y.~Pan's root-based algorithm for finding approximate \texttt{GCD} to the case where the polynomials are expressed in  Bernstein bases. We use the numerically stable companion pencil of \text{Gu{\dh}bj{\"o}rn} \text{J{\'o}nsson} to compute the roots, and the Hopcroft-Karp bipartite matching method to find the degree of the approximate \texttt{GCD}. We offer some refinements to improve the process.

\keywords{Bernstein basis \and Approximate \texttt{GCD} \and Maximum matching \and Bipartite graph \and Root clustering \and Companion pencil.}
\end{abstract}
\section{Introduction}
In general, finding the Greatest Common Divisor (\texttt{GCD}) of two exactly-known univariate polynomials is a well understood problem. However, it is also known that the \texttt{GCD} problem for \textit{noisy} polynomials (polynomials with errors in their coefficients) is an ill-posed problem. More precisely, a small error in coefficients of polynomials $P$ and $Q$ with a non-trivial \texttt{GCD} generically leads to a trivial \texttt{GCD}. As an example of such situation, suppose $P$ and $Q$ are non constant polynomials such that $P \vert Q$, then $\gcd(P,Q)= P$. Now for any $\epsilon>0$, $\gcd(P, Q+\epsilon) $ is a constant, since if $\gcd(P, Q+\epsilon) = g$, then $g \vert Q+\epsilon -Q = \epsilon$. This clearly shows that the \texttt{GCD} problem is an ill-posed one. We note that the choice of basis makes no difference to this difficulty.

At this point we have a good motivation to define something which can play a similar role to the \texttt{GCD} of two given polynomials 
which is instead well-conditioned. The idea is to define an \textit{approximate \texttt{GCD}}~\cite{botting2005using}. There are various definitions for approximate \texttt{GCD} which are used by different authors. All these definitions respect ``closeness'' and ``divisibility'' in some sense.

In this paper an approximate \texttt{GCD} of a pair of polynomials $P$ and $Q$ is the exact \texttt{GCD} of a corresponding pair of polynomials $\tilde{P}$ and $\tilde{Q}$ where $P$ and $\tilde{P}$ are ``close'' with respect to a specific metric, and similarly for $Q$ and $\tilde{Q}$ (see Definition \ref{def:approximategcd}).

Finding the \texttt{GCD} of two given polynomials is an elementary operation needed for many algebraic computations. Although in most applications the polynomials are given in the power basis, there are cases where the input is given in other bases such as the Bernstein basis. One important example of such a problem is finding intersection points of B\'ezier curves, which is usually presented in a Bernstein basis. For computing the intersections of B\'ezier curves and surfaces the Bernstein resultant and \texttt{GCD} in the Bernstein basis comes in handy (see~\cite{bini2006computing}).

One way to deal with polynomials in Bernstein bases is to convert them into the power basis. In practice poor stability of conversion from one basis to another and poor conditioning of the power basis essentially cancel the benefit one might get by using conversion to the simpler basis (see~\cite{FaroukiRajan}).

The Bernstein basis is an interesting one for various algebraic computations, for instance, see~\cite{mackey2016linearizations},~\cite{Victorpanmatching}.
There are many interesting results in approximate \texttt{GCD} including but not limited to~\cite{george1},~\cite{botting2005using},~\cite{Beckermann2018},~\cite{AAA},~\cite{zeng2011numerical},~\cite{FaroukiGoodman},~\cite{KarmarkarLkshman},~\cite{karmarkar1998approximate} and~\cite{kaltofen2007structured}. In~\cite{zhonggangAGCD}, the author has introduced a modification of the algorithm given by Corless, Gianni, Trager and Watt in~\cite{CorlessSVD}, to compute the approximate \texttt{GCD} in the power basis.

Winkler and Yang in~\cite{winkler} give an estimate of the degree of an approximate \texttt{GCD} of two polynomials in a Bernstein basis. Their approach is based on computations using 
resultant matrices. More precisely, they use the singular value decomposition of Sylvester and B\'ezout resultant matrices. We do not follow the approach of Winkler and Yang here, because they essentially convert to a power basis. 
Owing to the difference of results we do not give a comparison of our algorithm with the results of~\cite{winkler}.

Our approach is mainly to follow the ideas introduced by Victor Y.~Pan in~\cite{Victorpanmatching}, working in the power basis. In distinction to the other known algorithms for approximate \texttt{GCD}, Pan's method does not algebraically compute a degree for an approximate \texttt{GCD} first. Instead it works in a reverse way. In~\cite{Victorpanmatching} the author assumes the roots of polynomials $P$ and $Q$ are given as inputs. Having the roots in hand the 
algorithm generates a bipartite graph where one set of nodes contains the roots of $P$ and the other contains the roots of $Q$.
The criterion for defining the set of edges is based on Euclidean distances of roots. When the graph is defined completely, a 
matching algorithm will be applied. Using the obtained matching, a polynomial $D$ with roots as averages 
of paired close roots will be produced which is considered to be an approximate \texttt{GCD}. The last step is to use the roots of $D$ to replace the corresponding roots in $P$ and $Q$ to get $\tilde{P}$ and $\tilde{Q}$ as close polynomials.

In this paper we introduce an algorithm for computing approximate \texttt{GCD} in the Bernstein basis which relies on the above idea. For us the inputs are the coefficient vectors of $P$ and $Q$. We use the correspondence between the roots of a polynomial $f$ in a Bernstein basis and generalized eigenvalues of a corresponding matrix pencil $(A_f,B_f)$. This idea for finding the roots of $f$ was first used in~\cite{jonssonthesis}. Then by finding the generalized eigenvalues we get the roots of $P$ and $Q$ (see~\cite[Section $2.3$]{jonssonthesis}). Using the roots and similar methods to~\cite{Victorpanmatching}, we form a bipartite graph
and then we apply the maximum matching algorithm by Hopcroft and Karp~\cite{hopkroft} to get a maximum matching. Having
the matching, the algorithm forms a polynomial which is considered as an approximate \texttt{GCD} of $P$ and $Q$. The 
last step is to construct $\tilde{P}$ and $\tilde{Q}$ for which we apply a generalization of the method used in~\cite[Example 6.10]{corless2013} (see Section \ref{sec:apxpoly}). 

Note that our algorithm, like that of Victor Y.~Pan, does almost the reverse of the well-known algorithms for approximate \texttt{GCD}. Usually the 
algebraic methods do not try to find the roots. 
In~\cite{Victorpanmatching} Pan assumes the polynomials are represented by their roots. In our case we do not start with this assumption. Instead, by computing the roots we can then apply Pan's method.

The second section of this paper is provided some background for concrete computations with polynomials 
in Bernstein bases which is needed for our purposes. The third section present a method to construct a corresponding pair of
polynomials to a given pair $(P,Q)$. More precisely, this section generalizes the method mentioned in~\cite[Example 6.10]{corless2013} 
(which is introduced for power basis) in the Bernstein basis. The fourth section introduces a new algorithm for finding an
approximate \texttt{GCD}. In the final section we present numerical results based on our method.
\section{Preliminaries}\label{sec:preliminaries}
The Bernstein polynomials on the interval $0 \leq x \leq 1$ are defined as
\begin{equation}
B_{k}^{n}(x) = {n \choose k} x^k(1-x)^{n-k}
\end{equation}
for $k=0,\ldots,n$, where the binomial coefficient is as usual
\begin{equation}
{n\choose k}= \dfrac{n!}{k!(n-k)!} \>.
\end{equation}
More generally, in the interval $a \leq x \leq b$ (where $a < b$) we define
\begin{equation}
B_{a,b,k}^{n}(x) := {n \choose k} \dfrac{(x-a)^k(b-x)^{n-k}}{(b-a)^n}\>.
\end{equation}
When there is no risk of confusion we may simply write $B_k^{n}$'s for the $0 \leq x \leq 1$ case.
We suppose henceforth that $P(x)$ and $Q(x)$ are given in a Bernstein basis. 

There are various definitions for approximate \texttt{GCD}. The main idea behind all of them is to find ``interesting'' polynomials $\tilde{P}$ and $\tilde{Q}$ close to $P$ and $Q$ and use $\gcd(\tilde{P},\tilde{Q})$ as the approximate \texttt{GCD} of $P$ and $Q$. However, there are multiple ways of defining both ``interest'' and ``closeness''. To be more formal, consider the following weighted norm, for a vector $v$
\begin{equation}\label{eq:norm}
\left\|v\right \|_{\alpha,r} = \left( \sum_{k=1}^n \left| {\alpha_k v_k} \right|^r \right)^{1/r}
\end{equation}
for a given weight vector $\alpha\neq 0$ and a positive integer $r$ or $\infty$.
The map $\rho(u,v)= \left \| u-v \right\|_{\alpha,r}$ is a metric and we use this metric to compare the coefficient vectors of $P$ and $Q$.

In this paper we define an approximate \texttt{GCD} using the above metric or indeed any fixed semimetric. More precisely, we define the \textit{pseudogcd} set for the pair $P$ and $Q$ as
$$
A_{\rho} = \big\lbrace g(x) \;\; | \;\;  \exists \tilde{P}, \tilde{Q}\;\; \text{with}\;\; \rho(P,\tilde{P})\leq \sigma, \rho(Q,\tilde{Q}) \leq \sigma \;\; \text{and}\;\; g(x) = \gcd(\tilde{P},\tilde{Q}) \big\rbrace \>.
$$
Let 
\begin{equation}
d = \max_{\substack{g \in A_{\rho}}} \deg(g(x))\>.
\end{equation}
\begin{definition}\label{def:approximategcd}
An approximate \texttt{GCD} for $P,Q$ which is denoted by $\agcd{P}{Q}$, is $G(x) \in A_{\rho}$ where $\deg(G) = d$ and $\rho(P,\tilde{P})$ and $\rho(Q,\tilde{Q})$ are simultaneously minimal in some sense. For definiteness, we suppose that the maximum of these two quantities is minimized.   
\end{definition} 

In Section \ref{sec:RMP} we will define another (semi) metric, that uses roots. In Section \ref{sec:compappGCD} we will see that the parameter $\sigma$ helps us to find approximate polynomials such that the common roots of $\tilde{P}$ and $\tilde{Q}$ have at most distance (for a specific metric) $\sigma$ to the associated roots of $P$ and $Q$ (see Section \ref{sec:compappGCD} for more details).
\subsection{Finding roots of a polynomial in a Bernstein basis} \label{companion pencil}
In this section we recount the numerically stable method introduced by \text{Gu{\dh}bj{\"o}rn} \\
\text{J{\'o}nsson} for finding roots of a given polynomial in a Bernstein basis.
We only state the method without discussing in detail its stability and we refer the reader to~\cite{jonssonthesis} and~\cite{jonsson2004solving} for more details.

Consider a polynomial 
\begin{equation}
P(x)=\sum_{i=0}^n a_iB_i^{n}(x)
\end{equation}
in a Bernstein basis where the $a_i$'s are real scalars. We want to find the roots of $P(x)$ by constructing its companion pencil. In~\cite{jonssonthesis} J{\'o}nsson showed that this problem is equivalent to solving the following generalized eigenvalue problem.
That is, the roots of $P(x)$ are the generalized eigenvalues of the corresponding companion pencil to the pair
\begin{small}
\[
\textbf{A}_P= \left[ \begin{matrix}
-a_{n-1} & -a_{n-2} & \cdots & -a_1 & -a_0 \\
\\
1 & 0 & & \\
\\
& 1 & 0 & &\\
\\
&  & \ddots & \ddots & \\
\\
& & & 1 & 0\\ 
\end{matrix}
\right]   
, \,\,\,
\textbf{B}_P= \left[ \begin{matrix}
-a_{n-1}+\frac{a_n}{n} & -a_{n-2} & \cdots & -a_1 & -a_0 \\
\\
1 & \frac{2}{n-1} & & \\
\\
& 1 & \frac{3}{n-2} & &\\
\\
&  & \ddots & \ddots & \\
\\
& & & 1 & \frac{n}{1}\\ 
\end{matrix}
\right].
\]
\end{small}
That is, $P(x)=\det(x\textbf{B}_P-\textbf{A}_P)$. In~\cite{jonssonthesis}, the author showed that the above method is numerically stable. \\
\begin{theorem}~\cite[Section $2.3$]{jonssonthesis} \label{thm:jonsson}
Assume $P(x), \textbf{A}_P$ and $\textbf{B}_P$ are defined as above. $z$ is a root of $P(x)$ if and only if it is a generalized eigenvalue for the pair $(\textbf{A}_P,\textbf{B}_P)$.
\end{theorem}
\begin{proof}
We show
\begin{equation}
P(z)=0 \;\;\; \Leftrightarrow \;\;\; (z\textbf{B}_P-\textbf{A}_P)\left[ \begin{matrix}
B_{n-1}^n(z)(\frac{1}{1-z})\\
\vdots \\
B_{1}^n(z)(\frac{1}{1-z})\\
B_{0}^n(z)(\frac{1}{1-z})\\ 
\end{matrix}\right]=0 \>.
\end{equation}
We will show that all the entries of 
\begin{equation}\label{eq:genEigen}
(z\textbf{B}_P-\textbf{A}_P)\left[ \begin{matrix}
B_{n-1}^n(z)(\frac{1}{1-z})\\
\vdots \\
B_{1}^n(z)(\frac{1}{1-z})\\
B_{0}^n(z)(\frac{1}{1-z})\\ 
\end{matrix}\right]
\end{equation}
are zero except for possibly the first entry:
\begin{equation}\label{eq: eq7}
(z-1)B_1^n(z)(\frac{1}{z-1}) + nzB_0^n(z)(\frac{1}{z-1}) 
\end{equation}
since $B_1^n(z)=nz(1-z)^{n-1}$ and $B_0^n(z)=z^0(1-z)^n$ if $n \geq 1$, so equation \eqref{eq: eq7} can be written as
\begin{equation}
-nz(1-z)^{n-1} + nz\dfrac{(1-z)^n}{(1-z)} = -nz(1-z)^{n-1} + nz(1-z)^{n-1}=0 \>.
\end{equation}
Now for $k$-th entry:
\begin{equation}\label{eq: eq9}
\dfrac{(z-1)}{(1-z)}B_{n-k}^n(z) + \dfrac{k+1}{n-k}\dfrac{z}{(1-z)}B_{n-k-1}^n(z)
\end{equation}
Again we can replace $B_{n-k}^n(z)$ and $B_{n-k-1}^n(z)$ by their definitions. We find that equation \eqref{eq: eq9} can be written as
\begin{multline}
\dfrac{(z-1)}{(1-z)}{n \choose n-k}z^{n-k}(1-z)^{n-n+k} + \dfrac{k+1}{n-k}\dfrac{z}{(1-z)}{n \choose n-(k+1)} z^{\tiny {n-k-1}}(1-z)^{\tiny {n-(n-k-1)}} \\
 = -{n \choose n-k} z^{n-k}(1-z)^k + \dfrac{k+1}{n-k} \dfrac{n!}{(n-(k+1))!(k+1)!}z^{n-k}(1-z)^k \\
= \dfrac{n!}{(n-k)!k!}z^{n-k}(1-z)^k + \dfrac{n!}{(n-k)(n-(k+1))!k!} z^{n-k}(1-z)^k =0 \>.
\end{multline}
Finally, the first entry of equation \eqref{eq:genEigen} is
\begin{equation} \label{eq: eq11}
\dfrac{za_n}{n(1-z)}B_{n-1}^n(z) + \dfrac{a_{n-1}(z-1)}{(1-z)}B_{n-1}^n(z) + \sum_{i=0}^{n-2}a_iB_{n-1}^n(z)
\end{equation}
In order to simplify the equation \eqref{eq: eq11}, we use the definition of $B_{n-1}^n(z)$ as follows:
\begin{equation}
\dfrac{za_n}{n(1-z)}B_{n-1}^n(z) = \dfrac{z}{n}{n \choose n-1} z^{n-1}\dfrac{1-z}{1-z}
= \dfrac{z^n a_n}{n}{n \choose n-1}
= a_nB_n^n(z)
\end{equation}
So the equation \eqref{eq: eq11} can be written as:
\begin{equation}
a_nB_n^n(z) + (a_{n-1})B_{n-1}^n(z) + \sum_{i=0}^{n-2}a_iB_{n-1}^n(z)
\end{equation}
This is just $P(z)$ and so 
\begin{equation}
(z\textbf{B}_P-\textbf{A}_P)\left[ \begin{matrix}
B_{n-1}^n(z)(\frac{1}{1-z})\\
\vdots \\
B_{1}^n(z)(\frac{1}{1-z})\\
B_{0}^n(z)(\frac{1}{1-z})\\ 
\end{matrix}\right]=
\left[ \begin{matrix}
0\\
\vdots \\
0\\
0\\ 
\end{matrix}\right]
\end{equation}
if and only if $P(z)=0$.
\begin{flushright}
\QEDB
\end{flushright}
\end{proof}
This pencil (or rather its transpose) has been implemented in \mp since $2004$.

\begin{example}
Suppose $P(x)$ is given by its list of coefficients
\begin{equation}
[ 42.336, 23.058, 11.730, 5.377, 2.024]
\end{equation}
Then by using Theorem \ref{thm:jonsson}, we can find the roots of $P(x)$ by finding the eigenvalues of its corresponding companion pencil namely:
\begin{equation}
 \textbf{A}_P:= \left[ \begin {array}{cccc} -5.377& -11.73 &- 23.058&- 42.336
\\ \noalign{\medskip}1&0&0&0\\ \noalign{\medskip}0&1&0&0
\\ \noalign{\medskip}0&0&1&0\end {array} \right]
\end{equation}
and
\begin{equation}
\textbf{B}_P:= \left[ 
\begin {array}{cccc} - 4.871&- 11.73&-23.058&
- 42.336 \\ \noalign{\medskip} 1&
 .6666666666& 0& 0\\ \noalign{\medskip} 0& 1&
 1.5& 0\\ \noalign{\medskip} 0& 0& 1& 4
\end {array} \right] 
\end{equation}
Now if we solve the generalized eigenvalue problem using \mp for pair of $(\textbf{A}_P,\textbf{B}_P)$ we get:
\begin{equation}
\left[ \begin {array}{cccc}  5.59999999999989,  3.00000000000002, 2.1, 1.2 \end {array} \right] 
\end{equation}
Computing residuals, we have exactly\footnote{In some sense the exactness is accidental; the computed residual is itself subject to rounding errors. See~\cite{corless2013} for a backward error explanation of how this can happen.} $P(1.2)=0$, $P(2.1)=0$, $P(3)=0$, and $P(5.6)=0$ using de Casteljau's algorithm (see Section \ref{sec:deCasteljau}).
\end{example}
\subsection{Clustering the roots}
In this brief section we discuss the problem of having multiplicities greater than $1$ for roots of our polynomials. Since we are dealing 
with approximate roots, for an specific root $r$ of multiplicity $m$, we get $r_1, \ldots , r_m$ where $\vert r - r_i \vert \leq \sigma$ 
for $\sigma \geq 0$. Our goal in this section is to recognize the \textit{cluster}, $\lbrace r_1, \ldots , r_m \rbrace$, for a root $r$ 
as $\tilde{r}^m$ where $\vert \tilde{r}-r\vert \leq \sigma$ in a constructive way.\\
Assume a polynomial $f$ is given by its roots as $f(x) = \prod_{i = 1}^n (x -r_i)$. Our goal is to write $f(x) = \prod_{i = 1}^s (x -
t_i)^{d_i}$ such that $(x-t_i) \nmid f(x)/(x-t_i)^{d_i}.$ In other words, $d_i$'s are multiplicities of $t_i$'s. In order to do so we need a 
parameter $\sigma$ to compare the roots. If $\vert r_i -r_j \vert \leq \sigma$ then we replace both $r_i$ and $r_j$ with their average.

For our purposes, even the naive method, \textsl{i.e.} computing distances of all roots, works. This idea is presented as Algorithm \ref{alg:ClusterRoots}. It is worth mentioning that for practical purposes a slightly better way might be a modification of the well known divide and conquer algorithm for solving the closest pair problem in plane ~\cite[Section 33.4]{cormen2001introduction}.
\begin{algorithm}[H]
\begin{algorithmic}
\caption{\texttt{ClusterRoots}$(P, \sigma)$}
\label{alg:ClusterRoots}
\REQUIRE $P$ is a list of roots
\ENSURE $[(\alpha_1,d_1), \ldots, (\alpha_m,d_m)]$ where $\alpha_i$ is considered as a root with multiplicity $d_i$
\STATE $temp \gets Empty List$
\STATE $ C \gets Empty List$
\STATE $p \gets \texttt{size}(P)$ 
\STATE $i \gets 1$ 
\STATE while $i \leq p$ do\\
\hspace{.5 cm} $\texttt{append}(temp,P[i])$\\
\hspace{.5 cm} $j \gets i+1$\\
\hspace{.5 cm} while $j \leq p$ do\\
\hspace{1 cm} if $\vert P[i] -P[j] \vert \leq s$ then\\
\hspace{1.5 cm} $\texttt{append}(temp, P[j])$\\
\hspace{1.5 cm} $\texttt{remove}(P,j)$\\
\hspace{1.5 cm} $p \gets p-1$\\
\hspace{1 cm} else\\
\hspace{1.5 cm} $j \gets j+1$\\
\hspace{.5 cm} $i \gets i+1$\\
\hspace{.5 cm} $\texttt{append}(C,[\texttt{Mean}(temp),\texttt{size}(temp)])$\\
return $C$
\end{algorithmic}
\end{algorithm}
\subsection{The root marriage problem}\label{sec:RMP}
The goal of this section is to provide an algorithmic solution for solving the following problem:

\textit{\textbf{The Root Marriage Problem (RMP)}}: Assume $P$ and $Q$ are polynomials given by their roots. 
For a given $\sigma > 0$, for each root $r$ of $P$, (if it is possible) find a unique root of $Q$, say $s$, such that $\vert r-s \vert \leq \sigma$.

A solution to the RMP can be achieved by means of graph theory algorithms. We recall that a maximum matching for a bipartite graph $(V,E)$,
is $M \subseteq E$ with two properties:
\begin{itemize}
\item every node $v \in V$ appears as an end point of an edge in $E'$ at most once.
\item $E'$ has the maximum size among the subsets of $E$ satisfying the previous condition.
\end{itemize}

We invite the reader to consult~\cite{BondyMurty} and~\cite{West} for more details on maximum matching.

There are various algorithms for solving the maximum matching problem in a graph. Micali and Vazirani's matching algorithm is probably the 
most well-known. However there are more specific algorithms for different classes of graphs. In this paper, as in~\cite{Victorpanmatching}, we use the Hopcroft-Karp algorithm for
solving the maximum matching problem in a bipartite graph which has a complexity of $O((m+n)\sqrt{n})$ operations. 

Now we have enough tools for solving the RMP. The idea is to reduce the RMP to a maximum matching problem. In order to do so we have to 
associate a bipartite graph to a pair of polynomials $P$ and $Q$. For a positive real number $\sigma$, let $G^{\sigma}_{P,Q} = 
(G^{\sigma}_P \cup G^{\sigma}_Q, E^{\sigma}_{P,Q})$ where
\begin{itemize}
\item $G_P= \texttt{ClusterRoots}(\text{the set of roots of}\; P, \sigma)$,
\item $G_Q= \texttt{ClusterRoots}(\text{the set of roots of}\; Q, \sigma)$,
\item $E^{\sigma}_{P,Q} = \Big \lbrace \left(\{r,s\},\; \min(d_t,d_s) \right): \; r \in G_P, \; \text{with multiplicity }\; d_r, \; s \in G_Q \\ \;\; \;\;\;\;\;\;\;\; \;\;\;\;  \text{with multiplicity} \;\; d_s,\; \Big\vert r[1] -s[1] \Big\vert \leq \sigma \Big \rbrace$
\end{itemize}

Assuming we have access to the roots of polynomials, it is not hard to see that there is a naive algorithm to construct $G^{\sigma}_{P,Q}$ 
for a given $\sigma > 0$. Indeed it can be done by performing $O(n^2)$ operations to check the distances of roots where $n$ is the larger
degree of the given pair of polynomials. 

The last step to solve the RMP is to apply the Hopcroft-Karp algorithm on $G^{\sigma}_{P,Q}$ to get a maximum matching. The complexity of this 
algorithm is $O(n^{\frac{5}{2}})$ which is the dominant term in the total cost. Hence we can solve RMP in time $O(n^{\frac{5}{2}})$.

As was stated in Section \ref{sec:preliminaries}, we present a semi-metric which works with polynomial roots in this section. For
two polynomials $R$ and $S$, assume $m \leq n$ and $\lbrace r_1, \ldots , r_m \rbrace$ and $\lbrace t_1, \ldots, t_n \rbrace$ are respectively the sets of roots of $R$ and $T$. Moreover assume $S_n$ is the set of all permutations of $\lbrace 1, \ldots , n \rbrace$.
We define $$\rho(R,T) = \min_{\substack{\tau \in S_n}} \parallel [r_1 -t_{\tau(1)}, \ldots , r_m-t_{\tau(m)}] \parallel_{\alpha,r},$$
where $\alpha$ and $r$ are as before.
\begin{remark}
The cost of computing this semi-metric by this definition is $O(n!)$, and therefore prohibitive. However, once a matching has been found then
$$\rho(R, T)= \parallel [r_1 - s_{\text{match}(1)}, r_2 - s_{\text{match}(2)}, \ldots, r_m - s_{\text{match}(m)} ]\parallel _{\alpha , r}$$
where the notation $s_{\text{match}(k)}$ indicates the root found by the matching algorithm that matches $r_k$.
\end{remark}
\subsection{de Casteljau's Algorithm}\label{sec:deCasteljau}
Another component of our algorithm is a method which enables us to evaluate a given polynomial in a Bernstein basis at a given point. There are various methods for doing that. One of the most popular algorithms, for its convenience in Computer Aided Geometric Design (CAGD) applications and its numerical stability~\cite{farouki1987numerical}, is de Casteljau's algorithm which for convenience here is presented as Algorithm \ref{alg:de Cast}.
\begin{algorithm}
	\caption{\texttt{de Casteljau's Algorithm}}
	\label{alg:de Cast}
	\begin{algorithmic}[1]
		\REQUIRE C: a list of coefficients of a polynomial $P(x)$ of degree $n$ in a Bernstein basis of size $n+1$ \\
		\hspace{0.5cm} $\alpha$: a point
		\ENSURE $P(\alpha)$	\\
		\STATE $c_{0,j}$ $\gets$ $C_j$  for $j = 0 \ldots n.$
	    \STATE recursively define\\
	     $c_{i,j}$ $\gets$ $(1-\alpha)\cdot c_{i-1,j} + \alpha \cdot c_{i-1,j+1} $.\\
	     for $i = 1 \ldots n$ and $j = 1 \ldots n-i$.
	    \STATE return $c_{n,0}$.
	 	\end{algorithmic}
\end{algorithm}

We note that the above algorithm uses $O(n^2)$ operations for computing $P(\alpha)$. In contrast, Horner's algorithm for the power basis, Taylor polynomials, or the Newton basis, and the Clenshaw algorithm for orthogonal polynomials, and the barycentric forms\footnote{Assuming that the barycentric weights are precomputed.} for Lagrange and Hermite interpolational basis cost $O(n)$ operations.
\section{Computing Approximate Polynomials}\label{sec:apxpoly}
This section is a generalization of~\cite[Example 6.10]{corless2013} in Bernstein bases. The idea behind the algorithm is to create a linear system from coefficients of a given polynomial and the values of the polynomial at the approximate roots.

Now assume 
\begin{equation}
P(x) = \sum_{i=0}^n p_i B_{i}^{n}(x)
\end{equation}
is given with $\alpha_1, \ldots, \alpha_t$ as its approximate roots with multiplicities $d_i$. Our aim is to find 
\begin{equation}
\tilde{P}(x) = (P+\Delta P)(x)
\end{equation}
where 
\begin{equation}
\Delta P(x)= \sum_{i=0}^n (\Delta p_i)B_{i}^{n}(x)
\end{equation}
so that the set $\lbrace \alpha_1, \ldots , \alpha_t \rbrace$ appears as exact roots of $\tilde{P}$ with multiplicities $d_i$ respectively. On the other hand, we do want to have some
control on the coefficients in the sense that the new coefficients are related to the previous ones. Defining $\Delta p_i =
p_i \delta p_i $ (which assumes $p_i$'s are non-zero) yields 
\begin{equation}
\tilde{P}(x) = \sum_{i=0}^n (p_i + p_i \delta p_i) B_{i}^{n}(x)
\end{equation}

Representing $P$ as above, we want to find $\lbrace \delta p_i \rbrace_{i = 0}^n$. It is worth mentioning that with our assumptions, since perturbations of each coefficient, $p_i$ of $P$ are proportional to itself, if $p_i =0$ then $\Delta p_i = 0$. In other words we have assumed zero coefficients in $P$ will not be perturbed.

In order to satisfy the conditions of our problem we have 
\begin{equation}
\tilde{P}(\alpha_j) = \sum_{i=0}^n (p_i + p_i \delta p_i) B_{i}^{n}(\alpha_j)=0 \>,
\end{equation}
for $j = 1, \ldots , t$. Hence 
\begin{equation}
\tilde{P}(\alpha_j) = \sum_{i=0}^n p_i B_{i}^{n}(\alpha_j)+ \sum_{i=0}^{n-1} p_i \delta p_i B_{i}^{n}(\alpha_j)=0 \>,
\end{equation}
or equivalently 
\begin{equation}\label{eq:apppoly}
\sum_{i=0}^{n-1} p_i \delta p_i B_{i}^{n}(\alpha_j)=-P(\alpha_j) \>,
\end{equation}

Having the multiplicities, we also want the approximate polynomial
$\tilde{P}$ to respect multiplicities. More precisely, for $\alpha_j$, a root of $P$ of multiplicity $d_j$, we expect that $\alpha_j$ has multiplicity $d_j$
as a root of $\tilde{P}$. As usual we can state this fact by means of derivatives of $\tilde{P}$. We want 
\begin{equation}
\tilde{P}^{(k)}(\alpha_j) = 0 \; \text{for} \; 0 \leq k \leq d
\end{equation}

More precisely, we can use the derivatives of Equation \eqref{eq:apppoly} to write
\begin{equation}\label{eq:apppolyder}
\left( \sum_{i=0}^{n-1} p_i \delta p_i B_{i}^{n}\right)^{(k)}(\alpha_j)=-P^{(k)}(\alpha_j) \>.
\end{equation}

In order to find the derivatives in \eqref{eq:apppolyder}, we can use the differentiation matrix ${\bf D_B}$ in the Bernstein basis which is 
introduced in~\cite{amiraslani2018differentiation}. We note that it is a sparse matrix with only $3$ nonzero elements in each column~\cite[Section 1.4.3]{amiraslani2018differentiation}. 
So for each root $\alpha_i$, we get $d_i$ equations of the type \eqref{eq:apppolyder}. This gives us a linear system in the $\delta p_i$'s.
Solving the above linear system using the Singular Value Decomposition (\texttt{SVD}) one gets the desired solution.

Algorithm \ref{alg:apppoly} gives a numerical solution to the problem. For
an analytic solution for one single root see~\cite{rezvani2005nearest},~\cite{stetter1999nearest},~\cite{hitz1999efficient} and~\cite{hitz1998efficient}.
\begin{algorithm}[H]
\caption{\texttt{Approximate-Polynomial}$(P,L)$}
	\label{alg:apppoly}
	\begin{algorithmic}[1]
	\REQUIRE $P:$ list of coefficients of a polynomial of degree $n$ in a Bernstein basis\\
	\hspace{.5cm} $L:$ list of pairs of roots with their multiplicities.
	\ENSURE $\tilde{P}$ such that for any $(\alpha,d) \in L$, $(x-\alpha)^d \vert \tilde{P}$.
	\STATE $Sys \gets Empty List$
	\STATE $D_B \gets$ Differentiation matrix in the Bernstein basis of size $n+1$
	\STATE $X \gets \begin{bmatrix} x_1 & &\ldots & &x_{n+1} \end{bmatrix}^t$
	\STATE $T \gets \texttt{EntrywiseProduct}(\texttt{Vector}(P) , X)$
	\STATE for $(\alpha,d) \in L$ do\\
	\hspace{.5cm} $A \gets I_{n+1}$ \\
	\hspace{.5cm} for $i$ from 0 to $d-1$ do\\
	\hspace{1cm} $A \gets D_B\cdot A$\\
	\hspace{1cm} $eq \gets \texttt{DeCasteljau}(A\cdot T, \alpha) = -\texttt{DeCasteljau}(A\cdot \texttt{Vector}(P),\alpha)$\\
	\hspace{1cm} \texttt{append}$(Sys, eq)$\\
	\STATE Solve $Sys$ using \texttt{SVD} to get a solution with minimal norm (such as \ref{eq:norm}), and return the result.
	\end{algorithmic}
\end{algorithm}

Although Algorithm \ref{alg:apppoly} is written for one polynomial, in practice we apply it to both $P$ and $Q$ separately with the appropriate
lists of roots with their multiplicities to get $\tilde{P}$ and $\tilde{Q}$. 
\section{Computing Approximate \texttt{GCD}}\label{sec:compappGCD}
Assume the polynomials $P(x) = \sum_{i=0}^n a_iB_i^{n}(x)$ and $Q(x)= \sum_{i=0}^m b_iB_i^{m}(x)$ are given by their lists of coefficients
and suppose $\alpha \geq 0$ and $\sigma >0$ are given. Our goal here is to compute an approximate \texttt{GCD} of $P$ and $Q$ with respect to the given $\sigma$. 
Following Pan~\cite{Victorpanmatching} as mentioned earlier, the idea behind our algorithm is to match the close roots of $P$ and $Q$ and then based on this matching
find approximate polynomials $\tilde{P}$ and $\tilde{Q}$ such that their \texttt{GCD} is easy to compute. The parameter $\sigma$ is our main 
tool for constructing the approximate polynomials. More precisely, $\tilde{P}$ and $\tilde{Q}$ will be constructed such that their roots are respectively approximations of roots of $P$ and $Q$ with $\sigma$ as their error bound. In other words, for any root $x_0$ of $P$, $\tilde{P}$(similarly for $Q$) has a root $\tilde{x}_0$ such that $\vert x_0 - \tilde{x}_0 \vert \leq \sigma$.

For computing approximate \texttt{GCD} we apply graph theory techniques. In fact the parameter $\sigma$ helps us to define a bipartite graph 
as well, which is used to construct the approximate \texttt{GCD} before finding $\tilde{P}$ and $\tilde Q$. 

We can compute an approximate \texttt{GCD} of the pair $P$ and $Q$, which we denote by
$\agcd{P(x)}{Q(x)}$, in the following $5$ steps. \\
\\
\textbf{Step $1$. finding the roots:} Apply the method of Section \ref{companion pencil} to get $X=\left[ x_1,x_2,\ldots,x_n \right]$ , the set of all roots of $P$ and $Y=\left[ y_1,y_2,\ldots,y_m\right]$, the set of all roots of $Q$.  \\
\noindent
\textbf{Step $2$. forming the graph of roots $G_{P,Q}$:}  With the sets $X$ and $Y$ we form a bipartite graph, $G$, similar to~\cite{Victorpanmatching} which depends on parameter $\sigma$ in the following way:\\
If $\left| x_i-y_j \right| \leq 2\sigma$ for $i=1, \ldots, n$ and $j=1,\ldots,m$, then we can store that pair of $x_i$ and $y_j$.
\\
\textbf{Step $3$. find a maximum matching in $G_{P,Q}$:} Apply the Hopcroft-Karp algorithm~\cite{hopkroft} to get a maximum 
matching $\lbrace (x_{i_1},y_{j_1}), \ldots, (x_{i_r},y_{j_r})\rbrace$ where $1 \leq k \leq r$, $i_k\in \lbrace 1,\ldots,n \rbrace$
 and $j_k \in \lbrace 1,\ldots,m \rbrace$. 
 \\
 \\
\textbf{Step $4$. forming the approximate \texttt{GCD}:}
\begin{equation}
\agcd{P(x)}{Q(x)}=\prod_{s=1}^r (x-z_s)^{t_s}
\end{equation}
where $z_s=\dfrac{1}{2}(x_{i_s}+y_{j_s})$ and $t_s$ is the minimum of multiplicities of $x_s$ and $y_s$ for $1\leq s\leq r \>.$\\ \\
\textbf{Step $5$. finding approximate polynomials $\tilde{P}(x)$ and $\tilde{Q}(x)$:} Apply Algorithm \ref{alg:de Cast} with
$\lbrace z_1 , \ldots, z_r, x_{r+1} , \ldots , x_{n} \rbrace $ for $P(x)$ and $\lbrace z_1 , \ldots, z_r, y_{r+1} , \ldots , y_{m} \rbrace $ for $Q(x)$.

For steps $2$ and $3$ one can use the tools provided in Section \ref{sec:RMP}. We also note that the output of the above algorithm is directly related to the parameter $\sigma$ and an inappropriate $\sigma$ may
result in an unexpected result. \\

\section{Numerical Results}
In this section we show small examples of the effectiveness of our algorithm (using an implementation in \mp) with two low degree polynomials in a Bernstein basis, given by their list of coefficients:
\begin{align*}
P:= [ 5.887134, 1.341879, 0.080590, 0.000769,-0.000086]
\end{align*}
and 
\begin{align*}
Q:=[-17.88416,-9.503893,-4.226960,-1.05336]
\end{align*}
defined in \mp using Digits $:=30$ (we have presented the coefficients with fewer than $30$ digits for readability). So $P(x)$ and $Q(x)$ are seen to be 
\begin{align*}
P(x):= & 5.887134\, \left( 1-x \right) ^{4}+ 5.367516\, x \left( 1-x \right) ^{3} \\
 & +0.483544\,{x}^{2} \left( 1-x \right) ^{2}+ 0.003076\,{x}^{3}
 \left( 1-x \right) \\
 & - 0.000086\,{x}^{4}
\end{align*}
and
\begin{align*}
Q(x):= & - 17.88416\, \left( 1-x \right) ^{3}- 28.51168 \,x \left( 1-x \right) ^{2}\\
    & - 12.68088\,{x}^{2}\left( 1-x \right) - 1.05336\,{x}^{3}
\end{align*}

Moreover, the following computations is done using parameter $\sigma = 0.7$, and unweighted norm-$2$ as a simple example 
of Equation \eqref{eq:norm}, with $r =2$ and $\alpha=(1,\ldots,1)$.

Using Theorem \ref{thm:jonsson}, the roots of $P$ are, printed to two decimals for brevity, 
\[
\left[ \begin {array}{cccc}  5.3+ 0.0\,i, &  1.09+ 0.0\,i, &  0.99+ 0.0\,i, &  1.02+ 0.0\,i
\end {array} \right] 
\]

This in turn is passed to \texttt{ClusterRoots} (Algorithm \ref{alg:ClusterRoots}) to get
\[
P_{\texttt{ClusterRoots}}:=[[ 1.036+ 0.0\,i,3],[
 5.3+ 0.0\,i,1]]
\]
where $3$ and $1$ are the multiplicities of the corresponding roots.

Similarly for $Q$ we have:
\[
\left[ \begin {array}{ccc}  1.12+ 0.0\,i,
&  4.99+ 0.0\,i,
& 3.19+ 0.0\,i
\end {array} \right] 
\]
which leads to
\[
Q_{\texttt{ClusterRoots}}:=[[ 3.19+ 0.0\,i,1],[
 4.99+ 0.0\,i,1],[
 1.12+ 0.0\,i,1]]
\]
Again the $1$'s are the multiplicities of the corresponding roots.

Applying the implemented maximum matching algorithm in \mp (see Section \ref{sec:RMP}), a
maximum matching for the clustered sets of roots is 
\begin{align*}
T_{\texttt{MaximumMatching}}:=&[[ \left\{  4.99, 5.30 \right\} ,1], [ \left\{1.03 , 1.12 \right\} ,1]]
\end{align*}

This clearly implies we can define (see Step $4$ of our algorithm in Section \ref{sec:compappGCD})
\begin{align*}
\mathrm{agcd}^{0.7}_{\rho}(P,Q):= &(x-5.145)(x-1.078)
\end{align*}

Now the last step of our algorithm is to compute the approximate polynomials having these roots, namely
$\tilde{P}$ and $\tilde{Q}$. This is done using Algorithm \ref{alg:apppoly} which gives 
\begin{align*}
\tilde{P}:= &[ 6.204827, 1.381210, 0.071293, 0.000777,-0.000086]
\end{align*}
and
\begin{align*}
\tilde{Q}:= &[- 17.202067,- 10.003156,-4.698063,- 0.872077]
\end{align*}

Note that
\begin{align*}
\parallel P- \tilde{P} \parallel_{\alpha, 2}\; \approx 0.32 \leq 0.7 \;\; \text{and} \;\;\parallel Q- \tilde{Q} \parallel_{\alpha, 2} \; \approx 0.68 \leq 0.7
\end{align*}

We remark that in the above computations we used the built-in function \texttt{LeastSquares} in
\mp to solve the linear system to get $\tilde{P}$ and $\tilde{Q}$, instead of using the SVD ourselves. This equivalent method returns a solution to the system which 
is minimal according to norm-$2$. This can be replaced with any other solver which uses \texttt{SVD} to get
a minimal solution with the desired norm.

As the last part of experiments we have tested our algorithm on several random inputs of two polynomials of various degrees. The resulting polynomials $\tilde{P}$ and $\tilde{Q}$ are
compared to $P$ and $Q$ with respect to 2-norm (as a simple example of our weighted norm) and the root semi-metric which is defined in Section \ref{sec:RMP}. Some of the results are displayed in Table~\ref{tab:random}. 

\begin{small}
\begin{table}
\centering
\caption{Distance comparison of outputs and inputs of our approximate \texttt{GCD} algorithm on randomly chosen inputs.\label{tab:random}}
\begin{tabular}{|c|c|c|c|c|c|}
\hline 
&&&&&\\
$\max_{\deg}\lbrace P,Q \rbrace$ &  $\deg (agcd_{\rho}^{\sigma}(P,Q))$ & \hspace{.1 cm} $\parallel P-\tilde{P} \parallel_2$ \hspace{.1 cm} & \hspace{.1 cm} $\rho(P,\tilde{P})$ \hspace{.1 cm} & \hspace{.1 cm} $\parallel Q - \tilde{Q} \parallel_2$ \hspace{.1 cm} & \hspace{.1 cm} $\rho(Q,\tilde{Q})$ \hspace{.1 cm} \\
&&&&& \\
\hline
2 & 1 & 0.00473& 0.11619& 0.01199& 0.05820 \\
\hline
4 & 3 & 1.08900 & 1.04012 & 0.15880 & 0.15761 \\
\hline
6& 2& 0.80923& 0.75634& 0.21062& 0.31073 \\
\hline
7 & 2 & 0.02573& 0.04832 &  0.12336 & 0.02672\\
\hline
10 & 5 & 0.165979 & 0.22737 & 0.71190 & 0.64593\\
\hline
\end{tabular}
\end{table}
\end{small}
\section{Concluding remarks}
In this paper we have explored the computation of approximate \texttt{GCD} of polynomials given in a Bernstein basis, by using a method similar to that of Victor Y.~Pan~\cite{Victorpanmatching}. We first use the companion pencil of J{\'o}nsson to find the roots; we cluster the roots as Zeng does to find the so-called pejorative manifold. We then algorithmically match the clustered roots in an attempt to find $\mathrm{agcd}_{\rho}^{\sigma}$ where $\rho$ is the \textit{root distance semi-metric}. We believe that this will give a reasonable solution in the Bernstein coefficient metric; in future work we hope to present analytical results connecting the two.

%
%
 \bibliographystyle{splncs04}
%

\end{document}